\documentclass[DIV=12]{scrartcl}
\usepackage[utf8]{inputenc}
\usepackage[T1]{fontenc}
\usepackage{lmodern}

\title{On the uniqueness of collections of pennies and marbles}
\author{Sean Dewar \and Georg Grasegger \and Kaie Kubjas \and Fatemeh Mohammadi \and Anthony Nixon}
\date{}

\usepackage{graphicx}
\usepackage{amssymb,amsmath,amstext,amsfonts,amsthm}
\usepackage{float}
\usepackage{subcaption}
\graphicspath{ {./images/} }
\usepackage[dvipsnames,svgnames,table]{xcolor}
\usepackage[colorlinks=true,linkcolor=RoyalBlue,urlcolor=RoyalBlue,citecolor=PineGreen]{hyperref}
\usepackage{pgf,tikz}
\usetikzlibrary{calc}
\usetikzlibrary{backgrounds}
\usetikzlibrary{shapes.misc}
\usetikzlibrary{shapes.geometric}
\usepgflibrary{arrows.meta}
\usetikzlibrary{shapes.callouts}
\usetikzlibrary {shapes.arrows}
\usetikzlibrary {intersections}
\usetikzlibrary{angles,quotes}
\usepackage{booktabs}
\usepackage[nameinlink]{cleveref}
\usepackage[fixed]{fontawesome5}

\theoremstyle{plain}
\newtheorem*{theorem*}{Theorem}
\newtheorem{theorem}{Theorem}
\numberwithin{theorem}{section}

\newtheorem{lemma}[theorem]{Lemma}

\theoremstyle{definition}
\newtheorem{definition}[theorem]{Definition}

\newcommand{\R}{\mathbb{R}}

\newcommand{\edge}[2]{#1#2}

\newcounter{subfig}
\colorlet{colbg}{white}
\colorlet{colfg}{black}
\colorlet{colgraphv}{colfg!75!white}
\colorlet{colgraphe}{colfg!55!white}
\colorlet{colG}{DarkSeaGreen}
\definecolor{colR}{HTML}{CC6677}
\definecolor{colO}{HTML}{DDCC77}
\definecolor{colB}{HTML}{6699CC}
\colorlet{colY}{Gold!90!black}

\tikzstyle{vertex}=[fill=colgraphv,circle,inner sep=0pt, minimum size=4pt]
\tikzstyle{edge}=[line width=1.5pt,colgraphe]
\tikzstyle{backbone}=[edge,colR]
\tikzstyle{penny}=[colY,fill=colY!25!white] 
\tikzstyle{pennyg}=[colY!50!colR,fill=colY!50!colR!25!white] 
\tikzstyle{disk}=[dashed,colB,fill=colB!20!white,opacity=0.25] 
\tikzstyle{nodisk}=[dashed,colR,fill=colR!20!white,opacity=0.25]

\tikzstyle{labelsty}=[font=\scriptsize]

\tikzstyle{vertex2}=[fill=blue,circle,inner sep=0pt, minimum size=4pt]

\tikzstyle{env}=[rounded rectangle,minimum height=0.65cm,inner sep=0.3333em,align=left,minimum width=1cm,rounded corners=0pt,font=\small]
\tikzstyle{env1}=[env,rounded rectangle right arc=none,fill=colfg!50!colbg]
\tikzstyle{env2}=[env,rounded rectangle left arc=none]

\begin{document}

\maketitle
\begin{abstract}
In this note we study the uniqueness problem for collections of pennies and marbles. More generally, consider a collection of unit $d$-spheres that may touch but not overlap. Given the existence of such a collection, one may analyse the contact graph of the collection. In particular we consider the uniqueness of the collection arising from the contact graph. Using the language of graph rigidity theory, we prove a precise characterisation of uniqueness (global rigidity) in dimensions 2 and 3 when the contact graph is additionally chordal. We then illustrate a wide range of examples in these cases. That is, we illustrate collections of marbles and pennies that can be perturbed continuously (flexible), are locally unique (rigid) and are unique (globally rigid). We also contrast these examples with the usual generic setting of graph rigidity.
\end{abstract}

\section{Introduction}

Take a collection of same-size pennies and arrange them on a table. The contact graph of your collection has a vertex for each penny and there are edges precisely when the corresponding pennies touch. These so-called penny graphs are also known as unit coin graphs, minimum-distance graphs, smallest-distance graphs and closest-pair graphs in the wider literature \cite{EadesWhitesides,Hlineny}.
A well-studied but computationally difficult problem is determining which graphs can occur in this way; for example, it is NP-hard even for the special case of trees \cite{bowen2015realization}. See \Cref{fig:penny} for two examples of penny graphs: one with its corresponding penny graph representation and another with a non-penny realisation.

\begin{figure}[ht]
    \centering
    \begin{tikzpicture}[scale=1]
        \node[vertex] (a) at (0,0) {};
        \node[vertex] (b) at (1,0) {};
        \node[vertex,rotate around={60:(a)}] (d) at (b) {};
        \node[vertex] (c) at ($(d)+(1,0)$) {};
        \begin{scope}[on background layer]
            \foreach\n in {a,b,c,d}{\draw[penny] (\n) circle[radius=0.5cm];}
        \end{scope}
        \draw[edge] (a)edge(b) (b)edge(c) (c)edge(d) (d)edge(a) (b)edge(d);
        \node[colG] at ($(a)!0.5!(b)-(0,0.75)$) {\faCheck};
    \end{tikzpicture}
    \qquad
    \begin{tikzpicture}[scale=1]
        \node[vertex] (a) at (0,0) {};
        \node[vertex] (b) at (1,0) {};
        \node[vertex,rotate around={50:(a)}] (d) at (b) {};
        \begin{scope}[on background layer]
            \foreach\n in {a,b,d}{\draw[penny,opacity=0.5] (\n) circle[radius=0.5cm];}
        \end{scope}
        \draw[edge] (a)edge(b) (d)edge(a);
        \node[colR] at ($(a)!0.5!(b)-(0,0.75)$) {\faTimes};
    \end{tikzpicture}
    \caption{A graph with a penny realisation and a graph where the realisation does not fulfill the penny condition.}
    \label{fig:penny}
\end{figure}
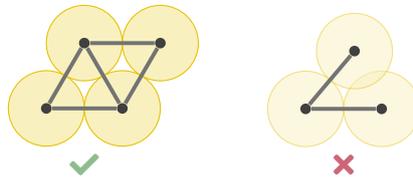

In this article we study the corresponding uniqueness problem. That is, given such a collection of pennies, one may always translate or rotate the entire collection, but are there other ways the shape of the collection can change, locally or globally?
We use rigidity theory to provide insight into this question and the corresponding 3-dimensional question for marbles. One motivation for this is the potential applications. Two examples are understanding the form and function of colloidal matter, see \cite{HolmesCerfon} for details, and identifiability in 3D genome reconstruction  \cite{cifuentes20233d}.

The basic objects of rigidity theory are \emph{frameworks}; pairs $(G,\rho)$ consisting of a graph $G=(V,E)$ and a map $\rho:V\rightarrow \mathbb{R}^d$. Two frameworks $(G,\rho)$ and $(G,\rho')$ in $\mathbb{R}^d$ are said to be \emph{equivalent} if for all $e\in E$, the length of $e$ in $(G,\rho)$ and $(G,\rho')$ is the same. More strongly, $(G,\rho)$ and $(G,\rho')$ are \emph{congruent} if $\rho$ can be obtained from $\rho'$ by applying a Euclidean isometry of $\mathbb{R}^d$.
The framework $(G,\rho)$ in $\mathbb{R}^d$ is \emph{rigid} if every equivalent framework $(G,\rho')$ in $\mathbb{R}^d$, with $\rho'$ in an open neighbourhood of $\rho$, is congruent to $(G,\rho)$. Moreover, $(G,\rho)$ in $\mathbb{R}^d$ is \emph{globally rigid} if every equivalent framework $(G,\rho')$ in $\mathbb{R}^d$ is congruent to $(G,\rho)$.
If a framework is not rigid, it is said to be \emph{flexible}.

Both the rigidity (assuming $d>1$) and the global rigidity of a framework
in $\mathbb{R}^d$ depend on the choice of realisation as well as the graph. However, if we restrict to generic frameworks, where $(G,\rho)$ is \emph{generic} if the set of coordinates of $\rho$ forms an algebraically independent set over $\mathbb{Q}$, then both properties depend only on the underlying graph \cite{AsimowRoth,GortlerHealyThurston}.
Because of this, we say a graph $G$ is \emph{generically (globally) rigid in $\mathbb{R}^d$} if any $d$-dimensional generic framework with graph $G$ is (globally) rigid.

In \Cref{sec:prelim}, we introduce the notions of rigidity and global rigidity for penny graphs. Since there is little additional complication, we work with the $d$-dimensional analogue of penny graphs, here called \emph{$d$-sphere graphs}, wherein pennies are replaced with $d$-dimensional unit spheres. Our first main contribution is \Cref{l:chordal}, which gives a complete characterisation of global rigidity for $d$-sphere graphs that are additionally chordal. The second main contribution of this note is the collection of examples presented in \Cref{sec:exam}. These examples, covering all but one possibility which we leave as an open problem, compare natural rigidity concepts for penny and marble graphs to their counterparts in generic rigidity theory.

\section{Rigid and globally rigid sphere graphs}
\label{sec:prelim}

We first introduce the fundamental object of study.

\begin{definition}
A \emph{$d$-sphere graph} is the contact graph of a collection of unit $d$-dimensional spheres with non-overlapping interiors, i.e., there exists a realisation $\rho\colon V\rightarrow \R^d$ such that
\begin{align*}
    \edge{v}{w}\in E \Rightarrow ||\rho(v)-\rho(w)||=1\\
    \edge{v}{w}\not\in E \Rightarrow ||\rho(v)-\rho(w)||>1;
\end{align*}
such a realisation $\rho$ is said to be a \emph{$d$-sphere graph realisation} of $G$. We 
use the terms penny graph and marble graph for the 2- and 3-dimensional cases respectively.
\end{definition}

We illustrate basic examples of rigid and flexible penny graphs in~\Cref{fig:rigflex-penny}.
This motivates the formal definition that follows.

\begin{figure}[ht]
    \centering
    \begin{tikzpicture}[scale=1]
        \node[vertex] (a) at (0,0) {};
        \node[vertex] (b) at (1,0) {};
        \node[vertex,rotate around={60:(a)}] (d) at (b) {};
        \node[vertex] (c) at ($(d)+(1,0)$) {};
        \begin{scope}[on background layer]
            \foreach\n in {a,b,c,d}{\draw[penny] (\n) circle[radius=0.5cm];}
        \end{scope}
        \draw[edge] (a)edge(b) (b)edge(c) (c)edge(d) (d)edge(a) (b)edge(d);
        \node[] at ($(a)!0.5!(b)-(0,0.75)$) {rigid};
    \end{tikzpicture}
    \qquad
    \begin{tikzpicture}[scale=1]
        \node[vertex] (a) at (0,0) {};
        \node[vertex] (b) at (1,0) {};
        \node[vertex,rotate around={90:(a)}] (d) at (b) {};
        \node[vertex] (c) at ($(d)+(1,0)$) {};
        \begin{scope}[on background layer]
            \foreach\n in {a,b,c,d}{\draw[penny] (\n) circle[radius=0.5cm];}
        \end{scope}
        \draw[edge] (a)edge(b) (b)edge(c) (c)edge(d) (d)edge(a);
        \node[] at ($(a)!0.5!(b)-(0,0.75)$) {flexible};
    \end{tikzpicture}
    \begin{tikzpicture}[scale=1]
        \node[vertex] (a) at (0,0) {};
        \node[vertex] (b) at (1,0) {};
        \node[vertex,rotate around={80:(a)}] (d) at (b) {};
        \node[vertex] (c) at ($(d)+(1,0)$) {};
        \begin{scope}[on background layer]
            \foreach\n in {a,b,c,d}{\draw[penny] (\n) circle[radius=0.5cm];}
        \end{scope}
        \draw[edge] (a)edge(b) (b)edge(c) (c)edge(d) (d)edge(a);
        \node[] at ($(a)!0.5!(b)-(0,0.75)$) {flexible};
    \end{tikzpicture}
    \caption{Rigid and flexible penny realisations of graphs.}
    \label{fig:rigflex-penny}
\end{figure}
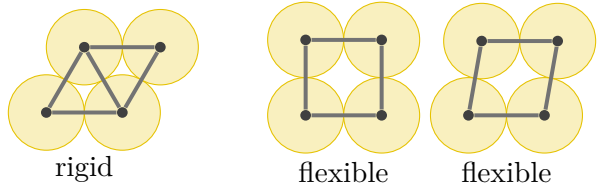

\begin{definition}
    Let $G$ be a $d$-sphere graph.
    We say that $G$ is \emph{$d$-sphere-rigid} if there exist finitely many $d$-sphere realisations modulo isometries.
    We say that $G$ is \emph{globally $d$-sphere-rigid} if there exists exactly one $d$-sphere graph realisation modulo isometries.
\end{definition}

Again we use (globally) penny-rigid and (globally) marble-rigid for the low-dimensional cases. 
An example of a globally penny-rigid graph is given in \Cref{fig:pennygrid}.

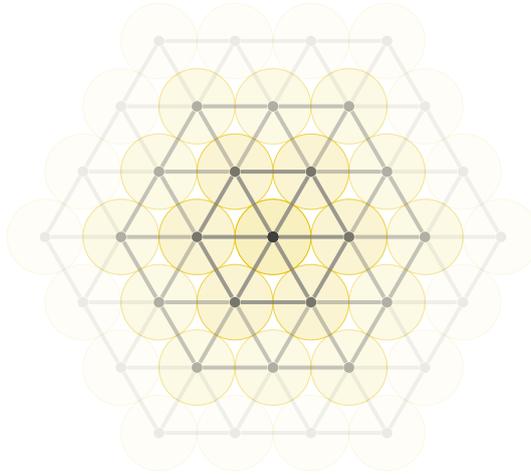
\begin{figure}[ht]
    \centering
    \begin{tikzpicture}
        \foreach \w in {1,2,...,6}
        {
            \node[vertex] (a0\w) at (0,0) {};
        }
        \begin{scope}[on background layer]
            \draw[penny] (a01) circle[radius=0.5cm];
        \end{scope}
        \foreach \level [%
            remember=\level as \levelo (initially 0),%
            evaluate=\level as \b using {int(6*\level)},%
            evaluate=\level as \step using {int(\level+1)},%
            evaluate=\level as \op using 1-3*\level/10%
        ] in {1,2,3}
        {
            \begin{scope}[opacity=\op]
                \foreach \w [evaluate=\w as \k using {int(mod(\w-1,\level))},evaluate=\w as \ww using {int(floor((\w-1)/\level)+1)}] in {1,2,...,\b}
                {
                    \node[vertex] (a\level\w) at ($(60*\ww:\level)+\k*(120+60*\ww:1)$) {};
                    \begin{scope}[on background layer]
                        \draw[penny,opacity=\op] (a\level\w) circle[radius=0.5cm];
                    \end{scope}
                }
                \foreach \w [remember=\w as \wo (initially \b)] in {1,2,...,\b}
                {
                    \draw[edge] (a\level\w)edge(a\level\wo);
                }
                \foreach \w [count=\i,evaluate=\w as \wo using {int(1+\levelo*(\i-1))}] in {1,\step,...,\b}
                {
                    \draw[edge] (a\level\w)edge(a\levelo\wo);
                    \ifnum\level>1
                        \foreach \k [%
                            evaluate=\k as \wk using {int(\w+\k)},%
                            evaluate=\k as \wok using {int(mod(\wo+\k-1,6*\levelo)+1)},%
                            evaluate=\k as \wokm using {int(\wo+\k-1)}
                        ] in {1,...,\levelo}
                        {
                            \draw[edge] (a\level\wk)edge(a\levelo\wok);
                            \draw[edge] (a\level\wk)edge(a\levelo\wokm);
                        }
                    \fi
                } 
            \end{scope}
        }
    \end{tikzpicture}
    \caption{A grid of pennies, which illustrates a family of examples of globally penny-rigid graphs.}
    \label{fig:pennygrid}
\end{figure}

Note that a $d$-sphere graph is $d$-sphere-rigid if and only if every $d$-sphere realisation of the graph is a rigid framework.
On the other hand, let $G$ be a $d$-sphere graph.
If there exists a $d$-sphere realisation of $G$ which is, as a framework, globally rigid, then $G$ is globally $d$-sphere rigid; however the converse is not true,
as shown later in the paper.

\begin{lemma}\label{lem:rotate}
    Let $G=(V,E)$ be a graph such that $G= G_1 \cup G_2$ where $G_1=(V_1,E_1)$, $G_2=(V_2,E_2)$ and $|V_1 \cap V_2| \leq d-1$.
    Let $(G,\rho)$ be a framework in $\mathbb{R}^d$ so that:
    \begin{enumerate}
        \item\label{p:rotate1} the points $\{\rho (v) : v \in V_1 \cap V_2\}$ are affinely independent,
        and
        \item\label{p:rotate2} for each $i \in \{1,2\}$, the affine span of the points $\{\rho (v) : v \in V_i\}$ has dimension at least $d-1$.
    \end{enumerate}
    Then $(G,\rho)$ is flexible.
\end{lemma}

\begin{proof}
    This folklore result can be proved by applying a rotation to the framework $(G_1,\rho|_{V_1})$ about a codimension 2 subspace containing $\{\rho (v) : v \in V_1 \cap V_2\}$.
\end{proof}

\begin{lemma}\label{p:basic}
    Let $G=(V,E)$ be a $d$-sphere graph.
    \begin{enumerate}
        \item\label{p:basic1} If $G$ is globally $d$-sphere-rigid, then $G$ is $d$-sphere-rigid.
        \item\label{p:basic2} If $G$ is $d$-sphere-rigid, $d \leq 3$ and $|V|\geq d+1$, then $G$ is $d$-connected.
    \end{enumerate}
\end{lemma}

\begin{proof}
    Property \ref{p:basic1} follows immediately from the definitions.

    Suppose for a contradiction that $G$ is $d$-sphere rigid but not $d$-connected.
    Fix $G_1=(V_1,E_1)$ and $G_2=(V_2,E_2)$ to be subgraphs of $G$ where $G= G_1 \cup G_2$ and $|V_1 \cap V_2| \leq d-1$.
    Fix $\rho$ to be a $d$-sphere realisation of $G$.
    When the hypotheses of \Cref{lem:rotate} are satisfied, the fact that $(G,\rho)$ is rigid is contradicted and property \ref{p:basic2} follows. Since point \ref{p:rotate1} of \Cref{lem:rotate} always holds when $d \leq 3$, it remains to check the case when $d=3$ and the affine span of the set $\{\rho (v) : v \in V_i\}$ has dimension $1$ for some $i \in \{1,2\}$. 
    In this case flexibility is obvious as there must exist a vertex of degree 1, contradicting $d$-sphere-rigidity.
\end{proof}

\begin{figure}[ht]
    \centering
  
    \begin{tikzpicture}
        \foreach \x in {1,2,...,5}
        {
            \node[vertex] (b\x) at (\x,0) {};
        }
        \foreach \x [evaluate=\x as \nx using int(\x+1)] in {1,2,...,4}
        {
            \node[vertex,rotate around={60:(b\x)}] (c\x) at (b\nx) {};
            \draw[edge] (b\x)edge(b\nx);
            \draw[edge] (b\x)edge(c\x);
            \draw[edge] (b\nx)edge(c\x);
        }
        \foreach \x [evaluate=\x as \nx using int(\x+1)] in {1,2,...,3}
        {
            \draw[edge] (c\x)edge(c\nx);
        }
        \node[vertex,rotate around={-60:(b1)}] (d1) at (b2) {};
        \node[vertex,rotate around={-60:(b4)}] (d2) at (b5) {};
        \draw[edge] (b2)edge(d1) (b1)edge(d1);
        \draw[edge] (b4)edge(d2) (b5)edge(d2);
        
        \node[vertex,rotate around={-60:(b1)}] (e1) at (d1) {};
        \node[vertex,rotate around={60:(b5)}] (e2) at (d2) {};
        \draw[edge] (b1)edge(e1) (d1)edge(e1);
        \draw[edge] (b5)edge(e2) (d2)edge(e2);
        
        \node[vertex,rotate around={-60:(e1)}] (f1) at (d1) {};
        \node[vertex,rotate around={60:(e2)}] (f2) at (d2) {};
        \draw[edge] (d1)edge(f1) (e1)edge(f1);
        \draw[edge] (d2)edge(f2) (e2)edge(f2);
        
        \node[vertex,rotate around={-60:(e1)}] (g1) at (f1) {};
        \node[vertex,rotate around={60:(e2)}] (g2) at (f2) {};
        \draw[edge] (e1)edge(g1) (f1)edge(g1);
        \draw[edge] (e2)edge(g2) (f2)edge(g2);
        
        \node[vertex,rotate around={60:(f1)}] (h1) at (g1) {};
        \node[vertex,rotate around={-60:(f2)}] (h2) at (g2) {};
        \draw[edge] (g1)edge(h1) (f1)edge(h1);
        \draw[edge] (g2)edge(h2) (f2)edge(h2);
        
        \node[vertex,rotate around={60:(f1)}] (i1) at (h1) {};
        \node[vertex,rotate around={-60:(f2)}] (i2) at (h2) {};
        \draw[edge] (f1)edge(i1) (h1)edge(i1);
        \draw[edge] (f2)edge(i2) (h2)edge(i2);
        
        \node[vertex] (l1) at ($(i1)+(1,0)$) {};
        \node[vertex] (l2) at ($(i2)-(1,0)$) {};
        \draw[edge] (i1)edge(l1) (l1)edge(l2) (l2)edge(i2);
        
        \node[vertex,rotate around={60:(l1)}] (m) at (l2) {};
        \draw[edge] (l1)edge(m) (l2)edge(m);
    \end{tikzpicture}
    \caption{A graph that is penny-rigid but not globally penny-rigid. The violation of global-penny rigidity comes from reflecting the penny corresponding to the degree two vertex through the unique line determined by the centers of its two adjacent pennies.}
    \label{fig:penny-rig-not-glob}
\end{figure}
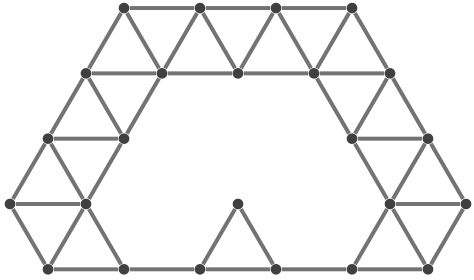

Note that both the converse implications are false. For example, in the penny case we observe that \Cref{fig:penny-rig-not-glob} is penny-rigid but not globally penny-rigid and any cycle with at least 4 vertices is 2-connected but not penny-rigid.

Before stating our main result, we need the following elementary lemma in which the clique number is the size of the largest complete subgraph. This follows from the observation that at most $d+1$ points in $d$-dimensional Euclidean space can be exactly distance 1 from each other.

\begin{lemma}\label{lem:clique}
Let $G$ be a $d$-sphere graph. Then the clique number of $G$ is at most $d+1$.    
\end{lemma}

Our next result significantly improves \Cref{p:basic} to give a complete characterisation of global $d$-sphere-rigidity in the case when $G$ is chordal. Recall that a graph is \emph{chordal} if every induced cycle is a triangle and a \emph{$d$-tree} if $G$ can be obtained recursively starting from $K_{d+1}$ and at each stage adding a new vertex of degree $d$ whose neighbour set induces a copy of $K_{d}$.

\begin{theorem}\label{l:chordal}
    Let $G=(V,E)$ be a chordal $d$-sphere graph and suppose $d\leq 3$.
    If $|V| \geq d+1$,
    then the following are equivalent.
    \begin{enumerate}
        \item\label{l:chordal1} $G$ is globally $d$-sphere-rigid.
        \item\label{l:chordal2} $G$ is $d$-sphere-rigid.
        \item\label{l:chordal3} $G$ is $d$-connected.
        \item\label{l:chordal4} $|E| = d|V|-\binom{d+1}{2}$.
        \item\label{l:chordal5} $G$ is a $d$-tree.
    \end{enumerate}
\end{theorem}

\begin{proof}
    Since $G$ is a $d$-sphere graph, \Cref{lem:clique} implies that it has a clique number of at most $d+1$.
    The equivalence of conditions \ref{l:chordal3}, \ref{l:chordal4} and \ref{l:chordal5} for chordal graphs with clique number at most $d+1$ is well known (see, for example, \cite{patil}).
    \Cref{p:basic} shows that  \ref{l:chordal1} implies \ref{l:chordal2}, and  \ref{l:chordal2} implies  \ref{l:chordal3}.
    Hence, it suffices to prove that \ref{l:chordal3} implies \ref{l:chordal1}.
    This is trivial if $|V| =d+1$ since this implies $G$ is $K_{d+1}$.
    
    Suppose that $|V| > d+1$ and \ref{l:chordal3} implies \ref{l:chordal1} holds for any chordal $d$-sphere graph with less than $|V|$ vertices.
    Since $G$ is $d$-connected,
    it is a $d$-tree;
    in particular,
    $G$ is formed from a $d$-tree $G'=(V',E')$ by adding a vertex $x$ of degree $d$ that is adjacent to a clique $\{x_1,\ldots,x_d\}$ of $G'$.
    Since $G'$ is also a $d$-tree with at least $d+1$ vertices,
    there exists a vertex $y$ of $G'$ such that $\{x_1,x_2,\dots, x_d,y\}$ is a clique of $G'$.
    Choose any two $d$-sphere realisations $\rho,\rho'$ of $G$ where $\rho(v)=\rho'(v)$ for each $v \in \{x_1,\ldots,x_d,y\}$ (since this vertex set is a clique of $G'$, this effectively quotients out the isometries of $\mathbb{R}^d$).
    Since $G'$ is globally $d$-sphere-rigid by the induction hypothesis,
    we have that $\rho(v) = \rho'(v)$ for all $v \neq x$.
    To be distance 1 from each point $\rho(x_i)$,
    either $\rho'(x) = \rho(x)$ or $\rho'(x)$ is the point found by reflecting $\rho(x)$ through the affine hyperplane defined by $\rho(x_1), \ldots,\rho(x_d)$.
    The latter position is already occupied by $\rho(y) = \rho'(y)$,
    and so $\rho'(x) = \rho(y)$ would contradict that $(G,\rho)$ is a $d$-sphere realisation.
    Hence, $\rho = \rho'$,
    which concludes the proof.
\end{proof}

\section{Penny and marble graph examples}
\label{sec:exam}

In this section, we focus on penny and marble graphs and present a number of concrete examples which illustrate that the notions of flexibility, rigidity and global rigidity for penny/marble graphs differ from the corresponding concepts for generic frameworks. 

A generic rigid framework in $\mathbb{R}^d$ on at least $d$ vertices must have at least $d|V|-\binom{d+1}{2}$ edges \cite{Maxwell}. \Cref{fig:sparsepenny} illustrates that there is no such bound for penny-rigid graphs. The indicated path would generically flex, but the penny structure forces the centers on the path to be collinear which holds the path tight.
This proves in particular that rigidity does not imply penny rigidity.

\begin{figure}[ht]
    \centering
    \begin{tikzpicture}
        \foreach \x in {1,2,...,6}
        {
            \node[vertex] (b\x) at (\x,0) {};
        }
        \foreach \x [evaluate=\x as \nx using int(\x+1)] in {1,2,4,5}
        {
            \node[vertex,rotate around={60:(b\x)}] (c\x) at (b\nx) {};
            \draw[edge] (b\x)edge(b\nx);
            \draw[edge] (b\x)edge(c\x);
            \draw[edge] (b\nx)edge(c\x);
        }
        \foreach \x [evaluate=\x as \nx using int(\x+1)] in {1,4}
        {
            \draw[edge] (c\x)edge(c\nx);
        }
        \node[vertex,rotate around={60:(b3)}] (c3) at (b4) {};
        \draw[edge,dotted] (b3)edge(b4);
        \draw[edge,dotted] (b3)edge(c3);
        \draw[edge,dotted] (b4)edge(c3);
        \draw[edge,dotted] (c2)edge(c3) (c3)edge(c4);
        
        \node[vertex,rotate around={-60:(b1)}] (d1) at (b2) {};
        \node[vertex,rotate around={-60:(b5)}] (d2) at (b6) {};
        \draw[edge] (b2)edge(d1) (b1)edge(d1);
        \draw[edge] (b5)edge(d2) (b6)edge(d2);
        
        \node[vertex,rotate around={-60:(b1)}] (e1) at (d1) {};
        \node[vertex,rotate around={60:(b6)}] (e2) at (d2) {};
        \draw[edge] (b1)edge(e1) (d1)edge(e1);
        \draw[edge] (b6)edge(e2) (d2)edge(e2);
        
        \node[vertex,rotate around={-60:(e1)}] (f1) at (d1) {};
        \node[vertex,rotate around={60:(e2)}] (f2) at (d2) {};
        \draw[edge] (d1)edge(f1) (e1)edge(f1);
        \draw[edge] (d2)edge(f2) (e2)edge(f2);
        
        \node[vertex,rotate around={60:(d1)}] (g1) at (f1) {};
        \node[vertex,rotate around={-60:(d2)}] (g2) at (f2) {};
        \draw[edge] (d1)edge(g1) (f1)edge(g1);
        \draw[edge] (d2)edge(g2) (f2)edge(g2);

        \node[vertex] (l1) at ($(g1)+(1,0)$) {};
        \node[vertex] (l2) at ($(g2)-(1,0)$) {};
        \draw[edge] (g1)edge(l1) (l2)edge(g2);
        \draw[edge,dotted] (l1)edge(l2);
    \end{tikzpicture}
    \caption{A class of (globally) penny-rigid graphs with arbitrarily fewer edges than $2|V|-3$.}
    \label{fig:sparsepenny}
\end{figure}
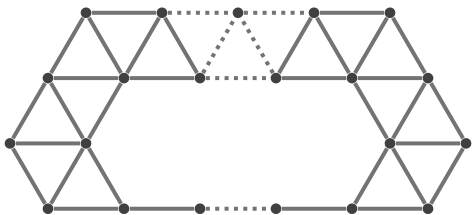

\Crefrange{fig:sparsepenny}{fig:penny-rig-notglob} illustrate the various options comparing the generic situation and the penny situation and \Crefrange{fig:MarbleGloballyRigid}{fig:flexmarble} are corresponding examples in the marble case. Data sets containing the marble realisations found in \Crefrange{fig:MarbleGloballyRigid}{fig:flexmarble} can be found at \cite{MarbleData}. The fact that all the examples in \Cref{fig:pennybarjoint} exhibit the bar-joint rigidity properties claimed is standard. The flexibility as penny graphs in (a) and (b) is due to the independent edge cuts being realised as parallel `bars'.

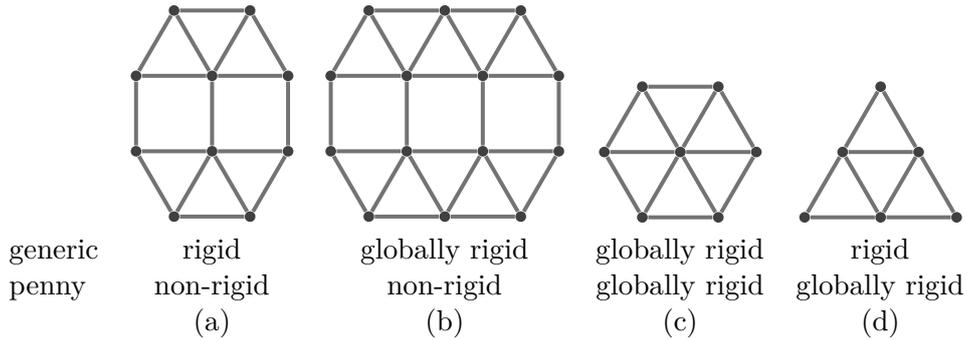
\begin{figure}[ht]
    \centering
    \begin{tabular}{lcccc}
        &
        \begin{tikzpicture}
            \foreach \x in {1,2,3}
            {
                \node[vertex] (b\x) at (\x,0.5) {};
                \node[vertex] (bp\x) at (\x,-0.5) {};
            }
            \foreach \x [evaluate=\x as \nx using int(\x+1)] in {1,2}
            {
                \node[vertex,rotate around={60:(b\x)}] (c\x) at (b\nx) {};
                \node[vertex,rotate around={-60:(bp\x)}] (cp\x) at (bp\nx) {};
                \draw[edge] (b\x)edge(b\nx);
                \draw[edge] (b\x)edge(c\x);
                \draw[edge] (b\nx)edge(c\x);
                \draw[edge] (bp\x)edge(bp\nx);
                \draw[edge] (bp\x)edge(cp\x);
                \draw[edge] (bp\nx)edge(cp\x);
            }
            \draw[edge] (c1)edge(c2) (cp1)edge(cp2);
            \draw[edge] (b1)edge(bp1) (b2)edge(bp2) (b3)edge(bp3);
        \end{tikzpicture}
        &
        \begin{tikzpicture}
            \foreach \x in {1,2,3,4}
            {
                \node[vertex] (b\x) at (\x,0.5) {};
                \node[vertex] (bp\x) at (\x,-0.5) {};
            }
            \foreach \x [evaluate=\x as \nx using int(\x+1)] in {1,2,3}
            {
                \node[vertex,rotate around={60:(b\x)}] (c\x) at (b\nx) {};
                \node[vertex,rotate around={-60:(bp\x)}] (cp\x) at (bp\nx) {};
                \draw[edge] (b\x)edge(b\nx);
                \draw[edge] (b\x)edge(c\x);
                \draw[edge] (b\nx)edge(c\x);
                \draw[edge] (bp\x)edge(bp\nx);
                \draw[edge] (bp\x)edge(cp\x);
                \draw[edge] (bp\nx)edge(cp\x);
            }
            \draw[edge] (c1)edge(c2) (c2)edge(c3) (cp1)edge(cp2) (cp2)edge(cp3);
            \draw[edge] (b1)edge(bp1) (b2)edge(bp2) (b3)edge(bp3) (b4)edge(bp4);
        \end{tikzpicture}
        &
        \begin{tikzpicture}
            \foreach \i [evaluate=\i as \w using 60*\i] in {1,2,3,4,5,6}
            {
                \node[vertex] (a\i) at (\w:1) {};
            }
            \node[vertex] (a0) at (0,0) {};
            \foreach \i [evaluate=\i as \oi using {int(mod(\i,6)+1)}] in {1,2,3,4,5,6}
            {
                \draw[edge] (a0)edge(a\i) (a\i)edge(a\oi);
            }
        \end{tikzpicture}
        &
        \begin{tikzpicture}[rotate=180]
            \coordinate (o) at (0,0);
            \coordinate (h) at (1,0);
            \coordinate[rotate around={60:(o)}] (r) at (h);
            \coordinate[rotate around={-60:(o)}] (l) at ($-1*(h)$);
            \foreach \i in {1,2,3}
            {
                \foreach \j in {1,...,\i}
                {
                    \node[vertex] (a\i\j) at ($\i*(r)+{\j-\i+1}*(h)$) {};
                }
            }
            \foreach \i [remember=\i as \li (initially 1)] in {2,3}
            {
                \foreach \j [remember=\j as \lj (initially 1)] in {2,...,\i}
                {
                    \draw[edge] (a\i\j)edge(a\i\lj) (a\i\j)edge(a\li\lj) (a\i\lj)edge(a\li\lj);
                }
            }
        \end{tikzpicture}
        \\
        generic & rigid & globally rigid & globally rigid & rigid\\
        penny  & non-rigid & non-rigid & globally rigid & globally rigid\\
        & (\alph{subfig})\stepcounter{subfig} & (\alph{subfig})\stepcounter{subfig} & (\alph{subfig})\stepcounter{subfig} & (\alph{subfig})
    \end{tabular}
    \caption{Graphs with different rigidity properties in the 2-dimensional generic and penny settings.}
    \label{fig:pennybarjoint}
\end{figure}

\begin{figure}[ht]
    \centering
    \begin{tikzpicture}[scale=1]
        \coordinate (o) at (0,0);
        \coordinate (h) at (1,0);
        \coordinate[rotate around={60:(o)}] (r) at (h);
        \coordinate[rotate around={-60:(o)}] (l) at ($-1*(h)$);
        \node[vertex] (a1) at (0,0) {};
        \node[vertex] (a2) at (h) {};
        \draw[edge] (a1)edge(a2);
        \foreach \i [count=\c,evaluate=\i as \j using int(\i+1),remember=\i as \o (initially 1)] in {3,5,7}
        {
            \node[vertex] (a\i) at ($\c*(r)$) {};
            \node[vertex] (a\j) at ($\c*(r)-(h)$) {};
            \draw[edge] (a\o)edge(a\i) (a\o)edge(a\j) (a\i)edge(a\j);
        }
        \draw[edge] (a2)edge(a3);
        \draw[edge] (a4)edge(a6) (a6)edge(a8);
        \foreach \i [count=\c,evaluate=\i as \j using int(\i+1),remember=\i as \o (initially 7)] in {9,11,13}
        {
            \node[vertex] (a\i) at ($(a7)+\c*(h)$) {};
            \node[vertex] (a\j) at ($(a5)+\c*(h)$) {};
            \draw[edge] (a\o)edge(a\i) (a\o)edge(a\j) (a\i)edge(a\j);
        }
        \node[vertex] (a16) at ($(a14)+(h)$) {};
        \draw[edge] (a5)edge(a10) (a10)edge(a12) (a12)edge(a14) (a14)edge(a16) (a13)edge(a16);
        \node[vertex] (a17) at ($(a16)-(r)$) {};
        \draw[edge] (a14)edge(a17) (a16)edge(a17);
        
        \path[save path=\circa,name path=circa] (a2) circle[radius=2cm];
        \path[save path=\circb,name path=circb] (a17) circle[radius=2cm];
        
        \path[name intersections={of=circa and circb, by={m1,m2} }];
        \node[vertex] (m) at (m2) {};
        \node[vertex] (b1) at ($(a2)!0.5!(m)$) {};
        \draw[edge] (a2)edge(b1) (b1)edge(m);
        \node[vertex,rotate around={-60:(a2)}] (b2) at (b1) {};
        \node[vertex,rotate around={-60:(b1)}] (b3) at (m) {};
        \draw[edge] (a2)edge(b2) (b2)edge(b1) (b2)edge(b3) (b3)edge(m) (b3)edge(b1);
        \node[vertex] (c1) at ($(a17)!0.5!(m)$) {};
        \draw[edge] (a17)edge(c1) (c1)edge(m);
        \node[vertex,rotate around={60:(a17)}] (c2) at (c1) {};
        \node[vertex,rotate around={60:(c1)}] (c3) at (m) {};
        \draw[edge] (a17)edge(c2) (c2)edge(c1) (c2)edge(c3) (c3)edge(m) (c3)edge(c1);
        \begin{scope}[opacity=0.5]
            \node[vertex,rotate around={60:(a2)}] (b2s) at (b1) {};
            \node[vertex,rotate around={60:(b1)}] (b3s) at (m) {};
            \draw[edge,dotted] (a2)edge(b2s) (b2s)edge(b1) (b2s)edge(b3s) (b3s)edge(m) (b3s)edge(b1);
            \draw pic ["\tiny71.76",draw=black,line width=1pt,angle radius=0.8cm,dotted] {angle=b2s--a2--a3};
            \draw pic ["\tiny68.68",draw=black,line width=1pt,angle radius=0.8cm,dotted] {angle=c1--m--b3s};
        \end{scope}
    \end{tikzpicture}
    \caption{A penny graph (with its corresponding penny realisation) that is generically rigid and penny-rigid,
    but is not generically globally rigid nor globally penny-rigid. The graph has exactly two penny realisations (modulo isometries). The dashed edges indicate the second realisation. The dotted angles show that this realisation is indeed a penny-realisation.}
    \label{fig:penny-rig-notglob}
\end{figure}
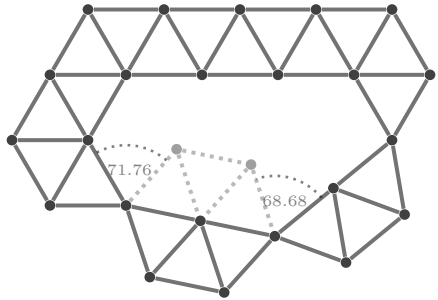

\begin{figure}[ht]
    \centering
    \includegraphics[width=4cm]{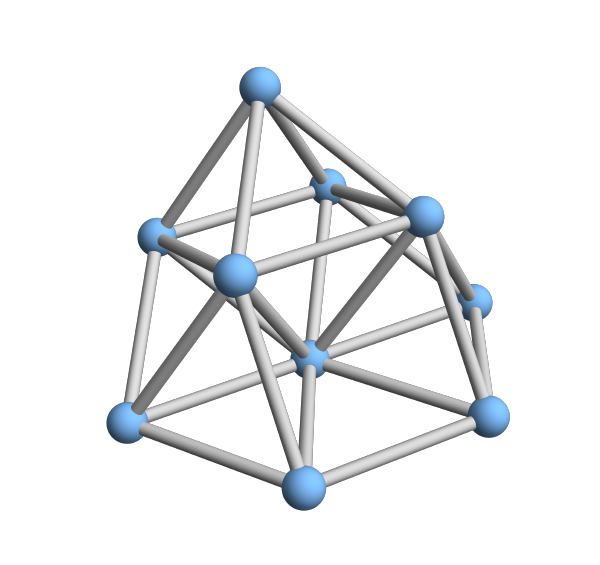}
    \includegraphics[width=4cm]{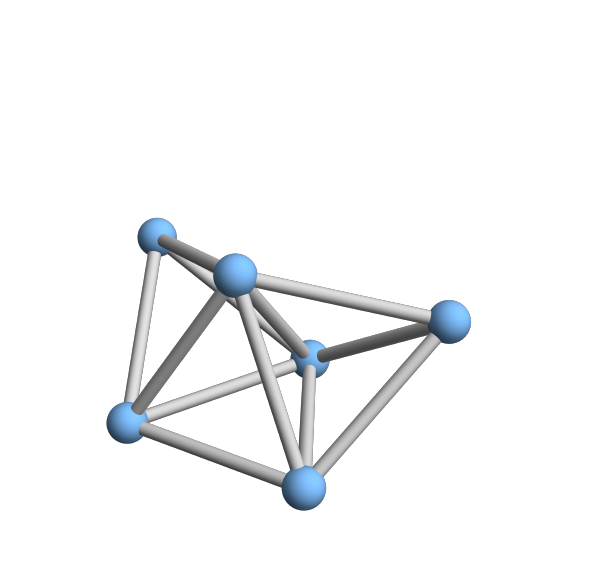}
    \includegraphics[width=5cm]{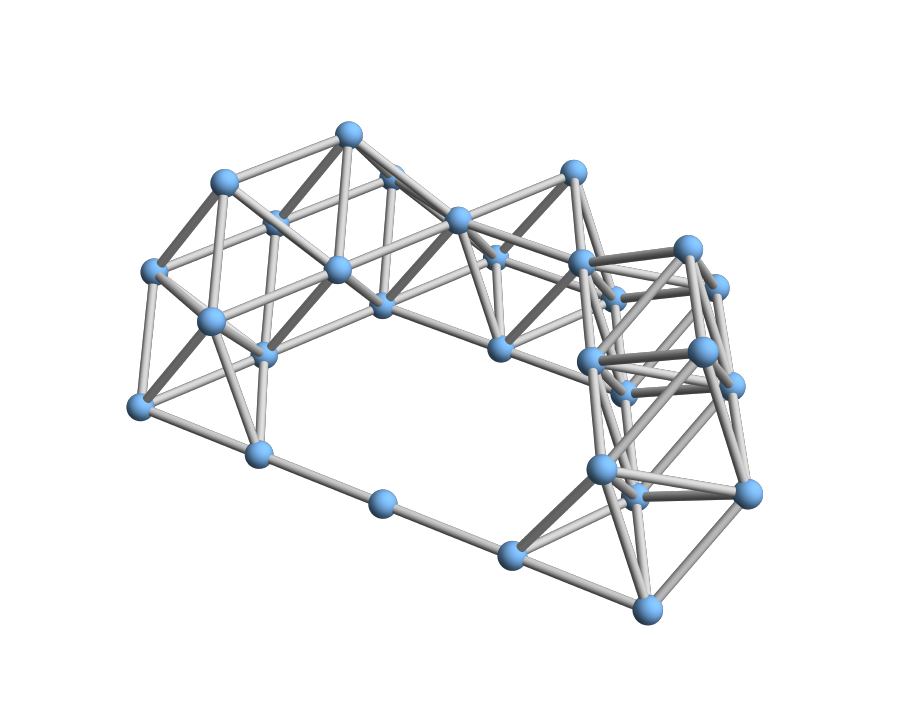}
    \caption{Three globally marble-rigid graphs with differing generic rigidity properties: (Left) globally rigid, (middle) rigid but not globally rigid, and (right) not rigid.
    }
    \label{fig:MarbleGloballyRigid}
\end{figure}

\begin{figure}[ht]
    \centering
    \includegraphics[width=5cm]{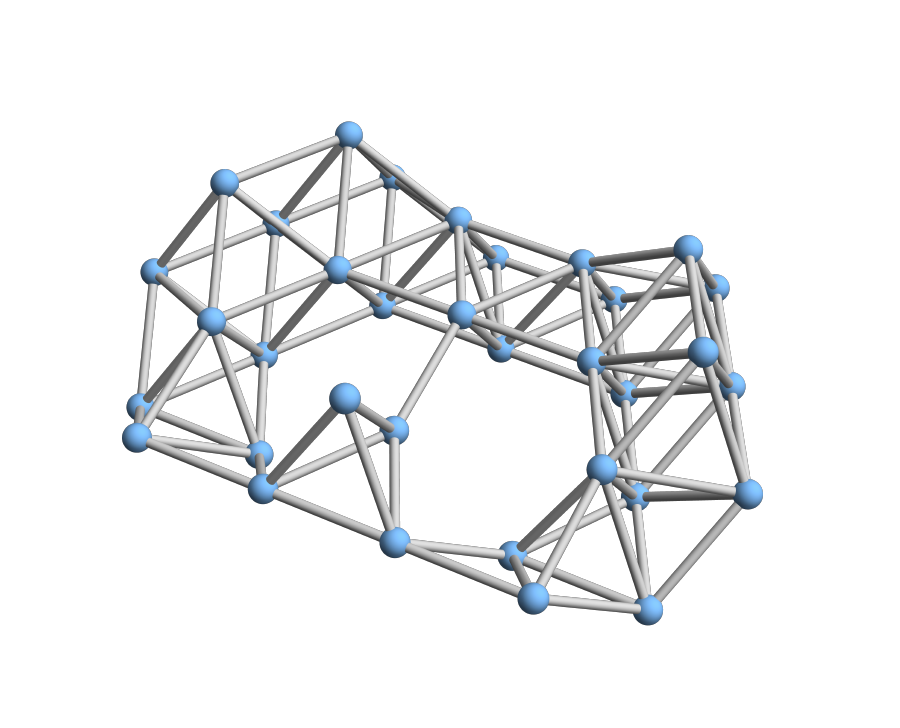}
    \caption{A graph that is marble rigid (not globally marble rigid) but not generically rigid.}
    \label{fig:BarJointFlexibleMarbleRigid}
\end{figure}

\begin{figure}[ht]
    \centering
    \includegraphics[width=5cm]{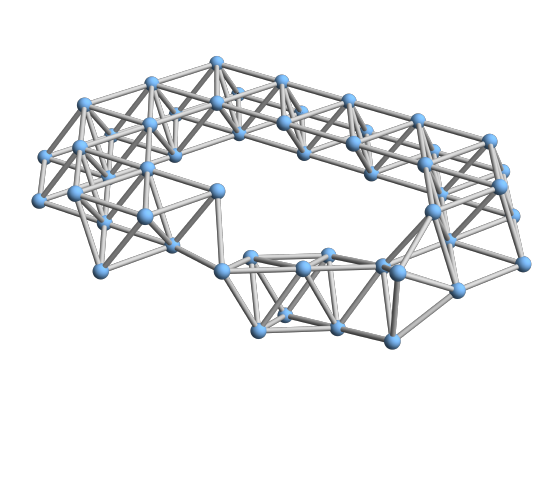}
    \includegraphics[width=5cm]{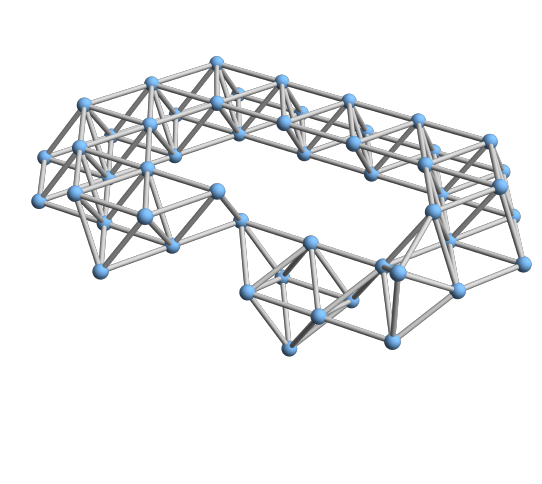}
    \caption{Two equivalent marble realisations of a graph that is marble-rigid and rigid, but not globally rigid or globally marble-rigid. Note that the smallest distance of non-adjacent vertices is just slightly larger than 1.}
    \label{fig:BarJointRigidMarbleRigid}
\end{figure}

\begin{figure}[ht]
     \centering
     \includegraphics[height=4cm]{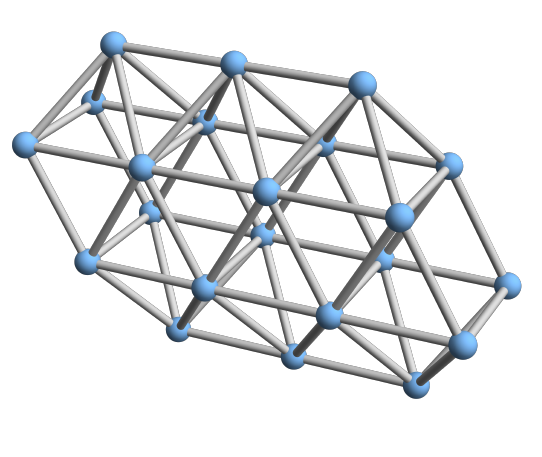}
     \includegraphics[height=4cm]{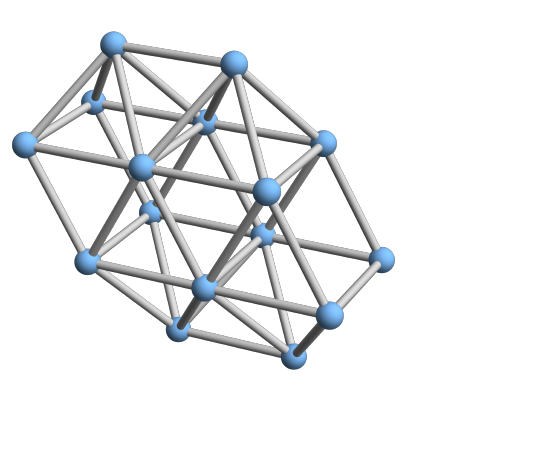}
     \caption{Two marble graphs that are not marble-rigid with differing generic rigidity properties: (left) globally rigid and (right) rigid but not globally rigid.}
     \label{fig:flexmarble}
\end{figure}

\begin{table}[ht!]
    \centering
    \begin{tabular}{|l|p{3.025cm}|p{3.025cm}|p{3.025cm}|}
    \hline
                              & Generically globally rigid & Generically rigid (not GGR) & Generically flexible\\
                               \hline
        Globally penny-rigid  & \Cref{fig:pennybarjoint} (c) & \Cref{fig:pennybarjoint} (d) & \Cref{fig:sparsepenny}\\
        \hline
        Penny-rigid (not GpR) &  OPEN & \Cref{fig:penny-rig-notglob} & \Cref{fig:penny-rig-not-glob}\\
        \hline
        Penny-flexible        & \Cref{fig:pennybarjoint} (b) & \Cref{fig:pennybarjoint} (a) & Path of length $\geq 2$\\
        \hline
    \end{tabular}
    \caption{Comparison of standard rigidity concepts and their variants for penny graphs.}
    \label{tab:barvspenny}
\end{table}

All of these examples are summarised in \Cref{tab:barvspenny} and \Cref{tab:barvsmarble}. As the tables show there is one option which we did not find an example for.
Specifically, it appears to be open whether there exists a generically globally rigid penny/marble graph that is penny/marble-rigid but not globally penny/marble-rigid.

\begin{table}[ht!]
    \centering
    \begin{tabular}{|l|p{3.025cm}|p{3.025cm}|p{3.025cm}|}
    \hline
                              & Generically globally rigid & Generically rigid (not GGR) & Generically flexible\\
                               \hline
        Globally marble-rigid  & \Cref{fig:MarbleGloballyRigid} (left) &  \Cref{fig:MarbleGloballyRigid} (center) & \Cref{fig:MarbleGloballyRigid} (right)\\
        \hline
        Marble-rigid (not GmR) &  OPEN & \Cref{fig:BarJointRigidMarbleRigid} & \Cref{fig:BarJointFlexibleMarbleRigid}\\
        \hline
        Marble-flexible        & \Cref{fig:flexmarble} (left) & \Cref{fig:flexmarble} (right) & Path of length $\geq3$\\
        \hline
    \end{tabular}
    \caption{Comparison of standard rigidity concepts and their variants for marble graphs.}
    \label{tab:barvsmarble}
\end{table}

\section{Conclusions and open problems}

We have initiated a study of rigidity theoretic concepts for penny and marble graphs and leave open many aspects of the theory for the interested reader.

While \Cref{l:chordal} provides a complete description of penny-rigidity and global penny-rigidity for chordal graphs, the general case is open. In the generic framework case, the celebrated results of Pollaczek-Geiringer \cite{Geiringer,Laman} and Jackson and Jord\'an \cite{JacksonJordan} give such characterisations. On the other hand, our characterisation for chordal graphs also applies in higher dimensions, whereas their results apply only in 2-dimensions. Indeed in higher dimensions it is believed to be a very challenging open problem to characterise generic (global) rigidity. Only a few special cases are understood.

A particularly nice instance is that of planar graphs. No planar graph (on at least 5 vertices) is globally rigid in $\mathbb{R}^3$ since they have at most $3|V|-6$ edges, hence they fail a well known necessary condition due to Hendrickson \cite{Hen}. On the other hand, a planar graph (on at least 3 vertices) is rigid in $\mathbb{R}^3$ if and only if it is a triangulation \cite{gluck1975almost}. It is conceivable that the case of planar graphs that are also marble graphs provides a tractable challenge.

\section*{Acknowledgements}
S.\,D.\ was supported by the Heilbronn Institute for Mathematical Research.
G.\,G.\ was partially supported by the Austrian Science Fund (FWF): P31888.
K.\,K.\ was partially supported by the Academy of Finland grant number 323416.
F.\,M.\ was partially supported by the FWO grants G0F5921N and G023721N, 
the KU Leuven iBOF/23/064 grant, and the UiT Aurora MASCOT project.
A.\,N.\ was partially supported by EPSRC grant number EP/W019698/1.

\bibliographystyle{plainurl}
\bibliography{references}

\bigskip 

\noindent
\footnotesize \textbf{Authors' addresses:}

\medskip 

\noindent{School of Mathematics, University of Bristol, Bristol, UK}
\hfill\texttt{sean.dewar@bristol.ac.uk}
\\
\noindent{Johann Radon Institute for Computational and Applied Mathematics}  \hfill \texttt{georg.grasegger@ricam.oeaw.ac.at}
\\
\noindent{Department of Mathematics and Systems Analysis, Aalto University} \hfill \texttt{kaie.kubjas@aalto.fi}
\\
\noindent{Departments of Mathematics and Computer Science, KU Leuven, Belgium} 
\hfill \texttt{fatemeh.mohammadi@kuleuven.be}
\\
\noindent{Mathematics and Statistics, 
Lancaster
University, Lancaster,
LA1 4YF, UK}
\hfill \texttt{a.nixon@lancaster.ac.uk}

\end{document}